\documentclass[11pt]{article}
\usepackage{geometry}                % See geometry.pdf to learn the layout options. There are lots.
\usepackage{graphicx}
\usepackage{amssymb}
\usepackage{amsthm}
\usepackage{epstopdf}
\newtheorem{defn}{Definition}[section]
\newtheorem{result}[defn]{Result}
\newtheorem{const}[defn]{Construction}
\newtheorem{thm}[defn]{Theorem}

\title{Overlap Cycles for Steiner Quadruple Systems}
\author{Victoria Horan\thanks{\texttt{vhoran@asu.edu}} and Glenn Hurlbert\thanks{\texttt{hurlbert@asu.edu}}\\ School of Mathematics and Statistics \\ Arizona State Universiy \\ Tempe, AZ  85287 USA}
\begin{document}
\maketitle

\begin{abstract}
	Steiner quadruple systems are set systems in which every triple is contained in a unique quadruple.  It is will known that Steiner quadruple systems of order $v$, or SQS($v$), exist if and only if $v \equiv 2, 4 \pmod 6$.  Universal cycles, introduced by Chung, Diaconis, and Graham in 1992, are a type of cyclic Gray code.  Overlap cycles are generalizations of universal cycles that were introduced in 2010 by Godbole.  Using Hanani's SQS constructions, we show that for every $v \equiv 2, 4 \pmod 6$ with $v > 4$ there exists an SQS($v$) that admits a 1-overlap cycle.  
\end{abstract}

\section{Introduction and Definitions}

A 3-$(v,4,1)$-design is known as a \textbf{Steiner quadruple system of order $v$}, or SQS$(v)$.  They may also be defined as a pair $(X, \mathcal{B})$ where $X$ is a set and $\mathcal{B}$ a collection of 4-element subsets of X called \textbf{quadruples}, with the property that any three points of $X$ are contained in a unique quadruple \cite{Cameron}.  The existence of Steiner quadruple systems was completely determined by H. Hanani in 1960.
\begin{thm}\label{general}
	\emph{\cite{Hanani}}  An SQS$(v)$ exists if and only if $v \equiv 2, 4 \pmod 6$.
\end{thm}

Ordering the blocks of a design and the points within its blocks is an important problem for many applications of design theory.  For example, in \cite{Disks} and \cite{Partial} the blocks of Steiner triple systems are ordered in a specific manner so as to produce efficient disk erasure correcting codes.  Other examples of these orderings include various Gray codes and universal cycles.  A \textbf{universal cycle} over a set of combinatorial objects represented as strings of length $n$ is an ordering that requires that the last $n-1$ letters of one string to match the first $n-1$ letters of its successor in the listing \cite{CDG}.

To prove Theorem \ref{general}, Hanani used six different recursive constructions and various base cases.  Using these constructions, we create 1-overlap cycles for each SQS$(v)$.  An \textbf{$s$-overlap cycle}, or \textbf{$s$-ocycle}, is a generalization of a universal cycle that relaxes the almost complete $n-1$ overlap between successive elements.  Instead of requiring that the last $n-1$ letters of one string match the first $n-1$ letters of its successor, an $s$-ocycle requires just the last $s$ letters of one to match the first $s$ letters of the next.  See \cite{Godbole} for some background on ocycles, including the construction of $s$-ocycles for $m$-ary words, and see \cite{OCSTS} for 1-ocycles and Steiner triple systems.  In this paper, we achieve the following result using Hanani's six recursive constructions.

\begin{thm}\label{main}
	For every $v \equiv 2, 4 \pmod 6$ with $v > 4$, there exists an SQS($v$) that admits a 1-ocycle.
\end{thm}

Note that $s$-ocycles may be thought of as a special type of Gray code.  Define a \textbf{$t$-swap Gray code} on a design $(X, \mathcal{B})$ to be an ordering of the blocks in $\mathcal{B}$ as $B_1, B_2, \ldots , B_n$ in which $|B_i \cap B_{i+1}| \geq t$.  Then an $s$-ocycle produces an $s$-swap Gray code.

We begin with a review of Hanani's constructions from \cite{Hanani} in Section 2, then illustrate 1-ocycles corresponding to each construction in Section 3, and finally conclude with several other results concerning $s$-ocycles for Steiner quadruple systems in Section 4.

\section{Hanani's Constructions}

To begin, we make a note that Hanani defines two systems of unordered pairs in \cite{Hanani}.  These systems, referred to as $P_\alpha(m)$ and $\overline{P}_\xi(m)$, are necessary for his constructions of Steiner quadruple systems.  However, our methods of constructing ocycles do not depend on the precise definitions of these sets, so we refer the reader to \cite{Hanani} for a complete definition, which will be omitted here.
\\

The first construction produces an SQS($2n$) from an SQS($n$).
\begin{const}\label{3.1}
	Let $(X, \mathcal{B})$ be an SQS with $|X| = n$.  Let $X = \{0, 1, \ldots, n-1\}$, and define a new point set $\mathcal{Y} = \{0, 1\} \times X$.  The blocks on $\mathcal{Y}$ that form an SQS($2n$) are as follows.
	\\
	\begin{enumerate}
		\item  $\{a_1x, a_2y, a_3z, a_4t\}$ where $\{x,y,z,t\} \in \mathcal{B}$ and $a_1 + a_2 + a_3 + a_4 \equiv 0 \pmod 2$
		\\
		\item  $\{0j, 0j', 1j, 1j'\}$ for $j \neq j'$
		\\
	\end{enumerate}
\end{const}

The second construction produces an SQS($3n-2$) from an SQS($n$).
\begin{const}\label{3.2}
	Let $(X, \mathcal{B})$ be an SQS with $|X| = n$.  Let $$X = \{0,1, 2, \ldots , n-2\} \cup \{A\}.$$  Let $\mathcal{B} = \mathcal{B}_A \cup \overline{\mathcal{B}_A}$ where $\mathcal{B}_A$ denotes all blocks containing $A$ and $\overline{\mathcal{B}_A}$ denotes all blocks \textit{not} containing the point A.  We construct new blocks on the set $$\mathcal{Y} = \left(\{0,1,2\} \times \{0,1,2, \ldots, n-2\} \right) \cup \{A\}.$$  Note that $\mathcal{Y}$ has cardinality $1 + 3(n-1) = 3n-2$.  The new blocks are as follows.
\\
\begin{enumerate}
	\item  $\{a_1 x, a_2 y, a_3z, a_4t\}$ for $\{x,y,z,t \} \in \overline{\mathcal{B}_A}$ and $a_1 + a_2 + a_3 + a_4 \equiv 0 \pmod 3$
	\\
	\item  $\{A, b_1u, b_2v, b_3w\}$ for $\{A, u, v, w \} \in \mathcal{B}_A$ and $b_1 + b_2 + b_3 \equiv 0 \pmod 3$
	\\
	\item  $\{iu, iv, (i+1)w, (i+2)w\}$ for $i \in \{0,1,2\}$ and $\{A, u, v, w\} \in \mathcal{B}_A$
	\\
	\item  $\{ij, ij', (i+1)j, (i+1)j'\}$ for $i \in \{0,1,2\}$ and $j,j' \in \{0,1,2, \ldots , n-2\}$ and $j \neq j'$
	\\
	\item  $\{A, 1j, 2j, 3j\}$ for $j \in \{0,1,2, \ldots n-2\}$
	\\
\end{enumerate}
\end{const}

The third construction produces an SQS($3n-8$) from an SQS($n$) when $n \equiv 2 \pmod {12}$.

\begin{const}\label{3.3}
	Let $(X, \mathcal{B})$ be an SQS with $|X| = n$ and $n \equiv 2 \pmod {12}$.  Let $$X = \{0, 1, 2, \ldots , n-5\}\cup \{ Ah : h \in \{0, 1, 2, 3\}\}.$$  We will make the assumption that $\{A0, A1, A2, A3\}$ is a block in $\mathcal{B}$.  Define $$\mathcal{Y} = \left( \{0, 1, 2\} \times \{0, 1, 2, \ldots , n-5\} \right) \cup \{Ah : h \in \{0, 1, 2, 3\}\}.$$  Note that $\mathcal{Y}$ has cardinality $3(n-4)+4 = 3n-8$.  We construct blocks on $\mathcal{Y}$ as follows.
\\
\begin{enumerate}
	\item  $\{A0, A1, A2, A3\}$
	\\
	\item  $\{ix, iy, iz, it\}$ where $\{x,y,z,t\} \in \mathcal{B} \setminus \{A0, A1, A2, A3\}$.  (If one of $x,y,z,t$ is $Ah$, omit the $i$.)  We denote this operation by $i \oplus (\mathcal{B} \setminus \{A0, A1, A2, A3\})$
	\\
	\item  $\{Aa_1, 0a_2, 1a_3, 2a_4\}$ where $a_1 + a_2 + a_3 + a_4 \equiv 0 \pmod {n-4}$
	\\
	\item  $\{(i+2)b_3, i(b_1+2k+1+i(4k+2)-d), i(b_1+2k+2+i(4k+2)+d),(i+1)b_2\}$ where $n-4  = 12k + 10$, $b_1+b_2+b_3 \equiv 0 \pmod {n-4}$, and $d \in \{0, 1, \ldots , 2k\}$
	\\
	\item  $\{ir_\alpha, is_\alpha, (i+1)r_\alpha', (i+1)s_\alpha'\}$ where $[r_\alpha, s_\alpha], [r_\alpha', s_\alpha'] \in P_\alpha (6k+5)$ (possible the same) where $\alpha = 4k+2, 4k+3, \ldots , 12k+8$
	\\
\end{enumerate}
\end{const}

The fourth construction produces an SQS($3n-4$) from an SQS($n$) when $n \equiv 10 \pmod {12}$.
\begin{const}\label{3.4}
	Let $(X, \mathcal{B})$ be an SQS with $|X| = n$ and $n \equiv 10 \pmod {12}$.  Let $$X = \{0, 1, 2, \ldots, n-3\} \cup \{A0, A1\}.$$  Define $$\mathcal{Y} = \left( \{ 0, 1, 2\} \times \{0, 1, \ldots , n-3\} \right) \cup \{A0, A1\}.$$  Note that $\mathcal{Y}$ has cardinality $3(n-2) + 2 = 3n-4$.  We construct blocks on $\mathcal{Y}$ as follows.
\\
\begin{enumerate}
	\item  $\{ix, iy, iz, it\}$ where $\{x,y,z,t\} \in \mathcal{B}$, or $i \oplus \mathcal{B}$.  (If one of $x,y,z,t$ is $Ah$, omit the $i$.)
	\\
	\item  $\{Aa_1, 0a_2, 1a_3, 2a_4\}$ where $a_1 + a_2 + a_3 + a_4 \equiv 0 \pmod {n-2}$ and $a_1 \in \{0, 1\}$ and $a_2, a_3, a_4 \in \{0,1,2, \ldots , n-2\}$
	\\
	\item  $\{(i+2)b_3, i(b_1+2k+1+i(4k+2)-d), i(b_1+2k+2+i(4k+2)+d), (i+1)b_2\}$ where $n = 12k + 10$, $b_1 + b_2 + b_3 \equiv 0 \pmod {n-2}$, and $d = 0, 1, \ldots , 2k$
	\\
	\item  $\{ir_\alpha, is_\alpha, (i+1)r_\alpha', (i+1)s_\alpha'\}$ where $[r_\alpha, s_\alpha], [r_\alpha', s_\alpha'] \in P_\alpha (6k+4)$ (possible the same) where $\alpha = 4k+2, 4k+3, \ldots , 12k+6$
	\\
	\end{enumerate}
\end{const}

The fifth construction produces an SQS($4n-6$) from an SQS($n$).
\begin{const}\label{3.5}
	Let $(X, \mathcal{B})$ be an SQS($n$) with $$X = \{0, 1, \ldots , n-2 \} \cup \{A0, A1\}.$$  Define $$\mathcal{Y} = \left( \{0, 1 \} \times \{0, 1\} \times \{0, 1, \ldots , n-3\} \right) \cup \{A0 ,A1\}.$$  Note that $\mathcal{Y}$ has cardinality $(2)(2)(n-2) + 2 = 4n-6$.  We construct blocks on $\mathcal{Y}$ as follows:
	\\
\begin{enumerate}
	\item  $h \oplus i \oplus \mathcal{B}$ where $h \in \{0, 1\}$ and $i \in \{0, 1\}$, and we ignore the prefix $hi$ from the points $A0, A1 \in X$
	\\
	\item  $\{A \ell, 00(2 c_1), 01(2c_2-\epsilon), 1 \epsilon (2c_3+\ell)\}$ where $\ell, \epsilon \in \{0, 1\}$ and $c_1 + c_2 + c_3 \equiv 0 \pmod k$, where $n=2k$
	\\
	\item  $\{A \ell, 00(2c_1+1), 01(2c_2-1-\epsilon), 1\epsilon (2c_3+1-\ell)\}$
	\\
	\item  $\{A \ell, 10(2c_1),11(2c_2-\epsilon),0\epsilon(2c_3+1-\ell)\}$
	\\
	\item  $\{A \ell, 10(2c_1+1),11(2c_2-1-\epsilon), 0 \epsilon (2c_3 + \ell)\}$
	\\
	\item  $\{h0(2c_1+\epsilon), h1(2c_2-\epsilon), (h+1)0 \overline{r}_{c_3}, (h+1)0 \overline{s}_{c_3}\}$ where $[\overline{r}_{c_3}, \overline{s}_{c_3}] \in \overline{P}_{c_3}(k)$ and $c_3 \in \{0, 1, \ldots , k-1\}$
	\\
	\item  $\{h0(2c_1-1+\epsilon), h1(2c_2-\epsilon), (h+1)1\overline{r}_{c_3}, (h+1)1\overline{s}_{c_3}\}$
	\\
	\item  $\{h0(2c_1+\epsilon), h1(2c_2-\epsilon), (h+1)1 \overline{r}_{k+c_3}, (h+1)1 \overline{s}_{k+c_3}\}$
	\\
	\item  $\{h0(2c_1-1 +\epsilon), h1(2c_2-\epsilon), (h+1)0\overline{r}_{k+c_3}, (h+1)0\overline{s}_{k+c_3}\}$
	\\
	\item  $\{h0r_\alpha, h0s_\alpha, h1r_\alpha',h1s_\alpha'\}$ where $[r_\alpha, s_\alpha], [r_\alpha', s_\alpha'] \in P_\alpha (k)$ and $\alpha \in \{0, 1, \ldots , n-4\}$
\end{enumerate}
\end{const}

The sixth and final construction produces an SQS($12n-10$) from an SQS($n$), and begins with the constructions of an SQS(14) and an SQS(38).
\\
\\
\textbf{SQS(14)} (listed as a 1-ocycle):
$$\begin{array}{l|l|l|l|l|l|l}
	2043 & 0BC2 & C841 & 38B5 & 0BD6 & B9CA & 259C \\
	3162 & 28B1 & 19B4 & 5AC3 & 6BC1 & A05B & CA68 \\
	25D3 & 1250 & 4095 & 36AD & 18D6 & B79D & 80B9 \\
	37A2 & 03B1 & 5134 & DBC3 & 629D & DA7C & 9158 \\
	28C3 & 1460 & 4265 & 36C4 & DB51 & CD40 & 8249 \\
	3560 & 0781 & 57A4 & 48B6 & 17C5 & 08A4 & 9368 \\
	07C3 & 19D0 & 48D5 & 60A7 & 51A6 & 4B07 & 87A9 \\
	38D0 & 0AC1 & 57D0 & 7196 & 69B5 & 749C & 9CD8 \\
	09A3 & 19A2 & 08C5 & 62C7 & 5C6D & CDA4 & 87CB \\
	39B2 & 2CD1 & 5BC4 & 73B6 & D59A & 4B2D & B17A \\
	2680 & 17D3 & 4783 & 64D7 & A26B & D872 & A964 \\
	0792 & 38A1 & 3D49 & 7586 & B34A & 27B5 & 4AD1 \\
	2AD0 & 139C & 9573 & 69C0 & A8DB & 58A2 & 1472
\end{array}$$
\\
\textbf{SQS(38)}:

We identify the points from the SQS(14) with the set $X = \{0, 1, 2, \ldots , 11\} \cup \{A0, A1\}$, and let $\mathcal{B}$ be the set of blocks form the SQS(14) on $X$.  Then we define our new point set $\mathcal{Y}$ as $$\mathcal{Y} = \left( \{0, 1, 2\} \times \{0, 1, 2, \ldots , 11\} \right) \cup \{A0, A1\}.$$  Note that $\mathcal{Y}$ has cardinality $(3)(12)+2 = 38$.  The blocks on $\mathcal{Y}$ are as follows:
\begin{enumerate}
	\item  $i \oplus \mathcal{B}$, where we omit the prefix $i$ for $A0$ or $A1$ \\
	\item  $\{Ah, 0b_1, 1b_2, 2(b_3+3h)\}$, where $b_1 + b_2 + b_3 \equiv 0 \pmod {12}$ and $h \in \{0, 1\}$ \\
	\item  $\{i(b_1+4+i), i(b_1+7+i), (i+1)b_2, (i+2)b_3\}$ \\
	\item  $\{ij, (i+1)(j+6 \epsilon), (i+2)(6 \epsilon-2j+1), (i+2)(6 \epsilon-2j-1)\}$ where $\epsilon \in \{0, 1\}$ \\
	\item  $\{ij, (i+1)(j+6 \epsilon), (i+2)(6 \epsilon-2j+2), (i+2)(6 \epsilon-2j-2)\}$ \\
	\item  $\{ij, (i+1)(j+6 \epsilon-3), (i+2)(6 \epsilon-2j+1), (i+2)(6 \epsilon-2j+2)\}$ \\
	\item  $\{ij, (i+1)(j+6\epsilon +3),(i+2)(6 \epsilon -2j -1), (i+2)(6 \epsilon -2j -2)\}$ \\
	\item  $\{ij, i(j+6), (i+1)(j+3\epsilon), (i+1)(j+6 + 3 \epsilon)\}$ \\
	\item  $\{i(2g+3\epsilon), i(2g+6+3\epsilon), i'( 2g+1), i'(2g+5)\}$ for $i' \neq i$ and $g \in \{0, 1, 2, 3, 4, 5\}$ \\
	\item $\{i(2g+3\epsilon), i(2g+6+3\epsilon), i'(2g+2), i'(2g+4)\}$ \\
	\item $\{ij, i(j+1), (i+1)(j+3e), (i+1)(j+3e+1)\}$ for $e = 0, 1, 2, 3$ \\
	\item $\{ij, i(j+2), (i+1)(j+3e), (i+1)(j+3e+2)\}$ \\
	\item $\{ij, i(j+4), (i+1)(j+3e), (i+1)(j+3e+4)\}$ \\
	\item $\{ir_\alpha, is_\alpha, i'r_\alpha', i's_\alpha'\}$ where $[r_\alpha, s_\alpha], [r'_\alpha, s'_\alpha] \in P_\alpha(6)$ for $\alpha = 4,5$
	\\
\end{enumerate}

\begin{const}\label{3.6}
	Let $(X, \mathcal{B})$ be an SQS with $|X| = n$.  Let $$X = \{B \} \cup \{0, 1, 2, \ldots , n-2\}.$$ Let $\mathcal{B}_B$ the subset of $\mathcal{B}$ with blocks containing $B$, and $\overline{\mathcal{B}_B}$ to be the complement.  Define $$\mathcal{Y} = \left( \{ 0, 1, 2, \ldots , n-2\} \times \{0, 1, 2, \ldots , 11\} \right) \cup \{A0, A1\}.$$  Note that $\mathcal{Y}$ has cardinality $(n-1)(12)+2 = 12n-10$.  We construct blocks on $\mathcal{Y}$ as follows:

\begin{enumerate}
	\item  $i \oplus \mathcal{B}(14)$ for $i \in \{0, 1, \ldots , n-2\}$ and where $i$ is omitted if the point is of the form $Ah$
	\\
	\item  \begin{enumerate}
	  \item  $\{Ah, ub_1, vb_2, w(b_3+3h)\}$ where $\{u,v,w,B\} \in \mathcal{A}_B$ and $b_1 + b_2 + b_3 \equiv 0 \pmod {12}$
\\
	  \item  $\{u\alpha_1, v\alpha_2, w\alpha_3, w \alpha_4\}$ where $\alpha_1, \alpha_2, \alpha_3, \alpha_4$ are the second indices (in order) of blocks of types (3) - (7) in the SQS(38)
\\
	  \item  $\{i\beta_1, i \beta_2, i'\beta_3, i'\beta_4\}$ where $\{i, i', B\}$ defines a unique block in $\mathcal{A}_B$ and $\beta_1, \beta_2, \beta_3, \beta_4$ are the second indices (in order) of blocks of type (8) - (14) in the SQS(38)
\\
	\end{enumerate}
	\item  $\{xa_1, ya_2, za_3, ta_4\}$, where $\{x,y,z,t\} \in \overline{\mathcal{A}_B}$ and $a_1 + a_2 + a_3 + a_4 \equiv 0 \pmod {12}$
\end{enumerate}
\end{const}

\begin{thm}\label{Hanani}
	For $n \equiv 2, 4 \pmod 6$ with $n>4$, there exists an SQS($n$).
\end{thm}
\begin{proof}
	Proceed by induction on $n$.  Since it is not possible to create a 1-ocycle for an SQS(2) or SQS(4), we begin our base cases with $n = 8, 10$.  If $n = 8$, we have the following SQS, arranged in 1-ocycle form:
	$$2148, 8523, 3684, 4578, 8156, 6287, 7813, 3576, 6471, 1572, 2163, 3274, 4135, 5462.$$    When $n= 10$, we have the following SQS(10), arranged in 1-ocycle form:
	$$\begin{array}{l|l|l|l|l|l}
		2145 & 6907 & 1372 & 7268 & 2694 & 7491 \\
		5263 & 7108 & 2483 & 8379 & 4681 & 1693 \\
		3674 & 8192 & 3594 & 9480 & 1583 & 3608 \\
		4785 & 2903 & 4506 & 0591 & 3705 & 8520 \\
		5896 & 3401 & 6157 & 1602 & 5297 & 0472
	\end{array}$$
	
	Before Construction \ref{3.6}, we illustrated an SQS(14) and an SQS(38).
	
	Let $n \geq 16$ with $n \equiv 4, 8 \pmod {12}$.  Then $n=2v$ for some $v \equiv 2,4 \pmod 6$.  Since $n \geq 16$, $v \geq 8$, and hence we use Construction \ref{3.1}.

	Let $n \geq 22$ with $n \equiv 4, 10 \pmod {18}$.  Then $n=3v-2$ for some $v \equiv 2,4 \pmod 6$.  Since $n \geq 22$, $v \geq 8$, and hence we use Construction \ref{3.2}.
	
	Let $n \geq 26$ with $n \equiv 2, 10 \pmod {24}$.  Then $n = 4v-6$ for some $v \equiv 2, 4 \pmod 6$.  Since $n \geq 26$, $v \geq 8$, and hence we use Construction \ref{3.5}.
	
	Let $n \geq 26$ with $n \equiv 26 \pmod {36}$.  Then $n = 3v-4$ for some $v \equiv 10 \pmod {12}$.  Since $n \geq 26$, $v \geq 10$, and hence we use Construction \ref{3.4}.
	
	Let $n \geq 34$ with $n \equiv 34 \pmod {36}$.  Then $n =3v-8$ for some $v \equiv 2 \pmod {12}$.  Since $n \geq 34$, $v \geq 14$, and hence we use Construction \ref{3.3}.
	
	Let $n \geq 86$ with $n \equiv 14, 38 \pmod {72}$.  Then $n = 12v-10$ for some $v \equiv 2, 4 \pmod 6$.  Since $n \geq 86$, $v \geq 8$, and hence we use Construction \ref{3.6}.

\end{proof}

\section{Ocycles and Hanani's Constructions}

Every result in this section will be proven in the same manner.  First, each type of block will generate individual ocycles.  Then we will illustrate how to connect all cycles together to form one 1-ocycle that covers all blocks.

\begin{result}\label{OC3.1}
	Let $(X, \mathcal{B})$ be an SQS of order $n$.  Then there exists an SQS of order $2n$ that admits a 1-ocycle.
\end{result}
\begin{proof}
We use Construction \ref{3.1}.  To create cycles on the quadruples we do the following.
	\begin{enumerate}
		\item  Fix $a_1, a_2 \in \{0, 1\}$ and $\{x,y,z,t\} \in \mathcal{B}$.  We have two choices for $a_3$, and this choice completely determines $a_4$.  Thus by fixing $a_1, a_2$ and $\{x,y,z,t\}$, we have identified two blocks of type (1).  We can construct the short 2-block 1-ocycle as shown below:
		$$a_1x \cdots a_2y \cdots a_1x.$$
		\item  For each $d \in \{1, 2, \ldots , \lfloor \frac{n}{2} \rfloor \}$, define the following cycle:
		$$\begin{array}{llll}
			00 & 10 & 1d & 0d \\
			0d & 1d & 1(2d) & 0(2d) \\
			0(2d) & 1(2d) & 1(3d) & 0 (3d) \\
			\vdots
		\end{array}$$
		We continue this cycle until we arrive back at $00$.  At this point, if gcd$(d,n)=1$, we will have covered all blocks of difference $d$.  However, if gcd$(d,n) > 1$, then we may need multiple cycles to cover all blocks.  In this case we just start anew with the first block missed.
		\\
	\end{enumerate}	
	To connect these cycles, we look to the blocks of type (2).  Note that when $d=1$ we will obtain one long cycle, call it $\mathcal{C}$, that has all points $0j$ with $j \in \{0, 1, 2, \ldots n-1\}$ as overlap points.  Then we can join all cycles from (2) to $\mathcal{C}$.  To attach the cycles of type (1), we utilize cycles from (1) with $a_1 = 0$.  We can connect these all to $\mathcal{C}$ using the overlap points $a_1x = 0x$.  For the remaining cycles with $a_1=1$, it is again possible to connect to $\mathcal{C}$ as follows.  If $a_2=0$ we can attach at the points $a_2y=0y$ on $\mathcal{C}$, and if $a_2=1$ we can attach at the points $a_2y=1y$ that exist on the cycles from (1) with $a_1=0,a_2=1$.
\end{proof}

\begin{result}\label{OC3.2}
	Let $(X, \mathcal{B})$ be an SQS of order $n$.  Then there exists an SQS of order $3n-2$ that admits a 1-ocycle.
\end{result}
\begin{proof}
	We use Construction \ref{3.2}.  To create cycles on the quadruples we do the following.
\begin{enumerate}
	\item  Fix $a_1, a_2 \in \{0,1,2\}$ and a block $\{x,y,z,t \} \in \overline{\mathcal{B}_A}$.  Note that these choices determine a set of 3 blocks in the SQS($3n-2$) since we may choose any $a_3 \in \{0,1,2\}$, but then our choices have completely determined $a_4$.  For each $a_1 \in \{0,1,2\}$ and $\{x,y,z,t\} \in \overline{\mathcal{B}_A}$, we create the 1-ocycle $$\begin{array}{ccccll} 
	0z & a_1 x & a_4t & 0y \\
	0y & 1z & a_4t & a_1 x & \hspace{5mm} & a_2=0 \\
	a_1x & 0y & a_4t & 2z \\
	\hline
	2z & 1y & a_4t & a_1x \\
	a_1 x & a_4t&1z & 1y & \hspace{5mm} & a_2 = 1 \\
	1y & a_4t& 0z& a_1 x \\
	\hline
	a_1x & 2z& a_4t & 2y \\
	2y & 1z& a_4t& a_1 x  & \hspace{5mm} & a_2 = 2\\
	a_1x & a_4t& 2y& 0z
	\end{array}$$
	\item  Fix $\{A, u, v, w \} \in \mathcal{B}_A$ and choose $b_1 \in \{0,1,2\}$.  This identifies 3 blocks: $$\{A, 0v, b_3^0w, b_1 u \}, \hspace{2 mm} \{A, 1v, b_3^1w, b_1 u\}, \hbox{ and } \{A, 2v, b_3^2w, b_1u\}$$  where $\{b_3^0, b_3^1, b_3^2\}=\{0,1,2\}.$  Note that the order of $\{b_3^0, b_3^1, b_3^2\}$ depends entirely on our choice of $b_1$.  We can string together the groups of 3 blocks for each choice of $b_1$ to create the following 1-ocycle:
	$$\begin{array}{ccccll}
		0v & A & 0w & 0u \\
		0u & 1v & 2w & A & \hspace{5mm} & b_1=0\\
		A &0u & 1w & 2v \\
		\hline
		2v & 1u & 0w & A \\
		A & 1v & 1w & 1u & \hspace{5mm} & b_1=1\\
		1u & 0v & 2w & A \\
		\hline
		A & 2v & 2w & 2u \\
		2u & 1v & 0w & A & \hspace{5mm} & b_1=2\\
		A & 2u & 1w & 0v
	\end{array}$$
	\item  Fix $\{A, u,v,w\} \in \mathcal{B}_A$ and $i \in \{0, 1, 2\}$.  The 3 blocks identified create a 1-ocycle: $$\begin{array}{cccc}
	iu & (i+1)w & (i+2)w & iv \\
	iv & (i+1)u & (i+2) u & iw \\
	iw & (i+1)v & (i+2)v & iu
	\end{array}$$
	\item  For each choice of $j$ and $j'$, we create the short cycle: $$\begin{array}{cccc}
		0j & 0j' & 1j' & 1j \\
		1j & 1j' & 2j' & 2j \\
		2j & 2j' & 0j' & 0j
	\end{array}$$

	\item  We can create the short strings:
	$$\begin{array}{cccc}
		0j & 1j & 2j & A \\
		A & 2(j+1) & 1(j+1) & 0(j+1) \\
	\end{array}$$
	To create an ocycle, we will utilize a few of the cycles from (4).  For each even $j \in \{0, 1, 2, \ldots , n-1\}$, set $j' = j+1$.  Then we modify the corresponding cycle from (4) to include one of the short strings as shown:
	$$\begin{array}{cccc}
		0j & 0j' & 1j' & 1j \\
		1j & 1j' & 2j' & 2j \\
		2j & 2j' & 0j' & 0j
	\end{array} \rightarrow
	\begin{array}{cccc}
		0j & 0j' & 1j' & 1j \\
		1j & 1j' & 2j' & 2j \\
		2j & 2j' & 0j & 0j' \\
		\hline
		0j' & 1j' & 2j' & A \\
		A & 2j & 1j & 0j
	\end{array}$$
	Note that this also ensures that $A$ appears as an overlap point, so together with $\mathcal{C}$ from (4), we are ensured that \textit{every} point appears as an overlap point.
\end{enumerate}
To connect these cycles and make one ocycle, we consider the blocks of type (3).  Fix $i \in \{0, 1, 2\}$ and $u \in \{0, 1, 2, \ldots , n-2\}$ and let $v$ vary through $\{0, 1, 2, \ldots , n-2\} \setminus \{u\}$.  Then each of these cycles from (3) containing $\{A, u, v\}$ has the point $iu$ as an overlap point so we can connect all of the cycles to make a long cycle, call it $\mathcal{C}_{i,u}$.  To connect $\mathcal{C}_{0,u}, \mathcal{C}_{1,u}$, and $\mathcal{C}_{2,u}$ we use a cycle from (4) with $j = u$.  Now we have created a cycle containing every point $ij \in \{0, 1, 2\} \times \{0, 1, 2, \ldots , n-2\}$ as an overlap point, so every other cycle can connect to this one.
\end{proof}

\begin{result}\label{OC3.3}
	Let $(X, \mathcal{B})$ be an SQS($n$) with $n \equiv 2 \pmod {12}$ that admits a 1-ocycle.  Then there exists an SQS($3n-8$) that admits a 1-ocycle.
\end{result}
\begin{proof}
	We use Construction \ref{3.3}.  We create cycles as follows on each set of blocks.
\begin{enumerate}
	\item  We will add this block to a cycle over blocks of type (2).
	\\
	\item  If we have an overlap cycle on $(X, \mathcal{B})$, call it $\mathcal{C}$, then for each $i \in \{0,1,2\}$ we get a cycle $\mathcal{C}_i$ by preceding each letter $x \in \{0, 1, 2, \ldots , n-5\}$ by $i$ and leaving all terms $Ah$ with $h \in \{0, 1, 2,3 \}$ unchanged.  However, each cycle uses the block $\{A0, A1, A2, A3\}$ so we can only use one of these cycles, say $\mathcal{C}_0$.  For $\mathcal{C}_1$ and $\mathcal{C}_2$, we get two strings by removing the block $\{A0, A1, A2, A3\}$.  Note that $\mathcal{C}_1$ and $\mathcal{C}_2$ can then be joined together at the endpoints.  For example, if the block appears as $A0, A1, A2, A3$ (in order) in $\mathcal{C}$, then we create the cycle $$A3 \hspace{5mm} \mathcal{C}_1 \hspace{5mm} A0 \hspace{5mm} \overline{\mathcal{C}_2} \hspace{5mm} A3$$ where $\overline{\mathcal{C}_2}$ is the string $\mathcal{C}_2$ listed in reverse order.  
	\\
	\item  If we fix $a_1, a_2$, then we have restricted our attention to $n-4$ distinct blocks.  Note that $n-4$ is even, and so we make the following cycles:
	$$\begin{array}{ccc}
		Aa_1 & \cdots & 0 a_2 \\
		0 a_2 & \cdots & Aa_1 \\
		Aa_1 & \cdots & 0 a_2 \\ 
		\vdots \\
		0a_2 & \cdots & Aa_1
	\end{array}$$
	\item  We will arrange each block so that the overlap points are $(i+2)b_3$ and $(i+1)b_2$.  Note that any choice of $b_2, b_3 \in \{0, 1, 2, \ldots n-5\}$ completely determines $b_1$, and by choosing $b_2, b_3, i, d$ we have identified a unique block.  Fix $b_2 \neq b_3$ and $ d$, and then we will create one cycle as follows.  The cycle when $b_2 \neq b_3$ is:
	$$\begin{array}{ccc}
		0b_3 & \cdots & 1b_2 \\
		1b_2 & \cdots & 2b_3 \\
		2b_3 & \cdots & 0b_2 \\
		0b_2 & \cdots & 1b_3 \\
		1b_3 & \cdots & 2b_2 \\
		2b_2 & \cdots & 0b_3 \\
	\end{array}$$
	When $b_2 = b_3$, we have the shorter cycle:
	$$\begin{array}{ccc}
		0b_2 & \cdots & 1b_2 \\
		1b_2 & \cdots & 2b_2 \\
		2b_2 & \cdots & 0b_2 \\
	\end{array}$$
	\\
	\item  Fix $\alpha$, and arbitrarily order the set $P_\alpha (6k+5) = \{[r(1),s(1)], [r(2), s(2)], \ldots , [r(t), s(t)]\}$.  For each $d = 0, 1, 2, \ldots , t-1$, we create a cycle $C_d$ as follows:
	$$\begin{array}{llll}
		0r(1) & 0s(1) & 1s(1+d) & 1r(1+d) \\
		1r(1+d) & 1s(1+d) & 2s(1+2d) & 2r(1+2d) \\
		2r(1+2d) & 2s(1+2d) & 0s(1+3d) & 0r(1+3d) \\
		\vdots \\
		2r(1-d) & 2s(1-d) & 0s(1) & 0r(1)
	\end{array}$$
	This cycle will always eventually reach the final block.  However, it may or may not also cover the intermediary blocks $\{1r(1-d), 1s(1-d), 2s(1), 2r(1)\}$ and $\{0r(1-d), 0s(1-d), 1s(1), 1r(1)\}$.  If $C_d$ misses these intermediary blocks, then we create a new cycle similar to $C_d$ but replacing $0,1,2$ with either $1,2,0$ or $2,0,1$, respectively. 

\end{enumerate}

To connect all of these cycles, we first focus on (4).  We create one long cycle by fixing $b_2=0$ and creating a cycle $C_{b_3}$ for each $b_3 \in \{0, 1, 2, \ldots , n-5\}$.  We connect all of these cycles at their points $00$ to create one long cycle $C$ that has every overlap of the type $$\{0, 1, 2\} \times \{0, 1, 2, \ldots , n-5\}.$$  Using this cycle $C$, we can connect every other cycle to $C$ to make one long cycle containing all quadruples.
\end{proof}

\begin{result}
	Let $(X, \mathcal{B})$ be an SQS($n$) with $n \equiv 10 \pmod {12}$ that admits a 1-ocycle.  Then there exists an SQS($3n-4$) that admits a 1-ocycle.
\end{result}

\begin{proof}
	We use Construction \ref{3.4}.  To form cycles, we do the following for each type.
\begin{enumerate}
	\item  If $\mathcal{C}$ is a cycle on $(X, \mathcal{B})$, then $i \mathcal{C}$ is a cycle on quadruples of type (1) for each $i \in \{0,1,2\}$.
	\\
	\item  If we fix $a_1, a_2$, then we have restricted our attention to $n-2$ distinct blocks.  Note that $n-2$ is even, and so we make the following cycles:
	$$\begin{array}{ccc}
		Aa_1 & \cdots & 0 a_2 \\
		0 a_2 & \cdots & Aa_1 \\
		Aa_1 & \cdots & 0 a_2 \\ 
		\vdots \\
		0a_2 & \cdots & Aa_1
	\end{array}$$
	\item  We will arrange each block so that the overlap points are $(i+2)b_3$ and $(i+1)b_2$.  Note that any choice of $b_2, b_3 \in \{0, 1, 2, \ldots n-3\}$ completely determines $b_1$, and by choosing $b_2, b_3, i, d$ we have identified a unique block.  Fix $b_2 \neq b_3$ and $ d$, and then we will create one cycle as follows.  The cycle when $b_2 \neq b_3$ is:
	$$\begin{array}{ccc}
		0b_3 & \cdots & 1b_2 \\
		1b_2 & \cdots & 2b_3 \\
		2b_3 & \cdots & 0b_2 \\
		0b_2 & \cdots & 1b_3 \\
		1b_3 & \cdots & 2b_2 \\
		2b_2 & \cdots & 0b_3 \\
	\end{array}$$
	When $b_2 = b_3$, we have the shorter cycle:
	$$\begin{array}{ccc}
		0b_2 & \cdots & 1b_2 \\
		1b_2 & \cdots & 2b_2 \\
		2b_2 & \cdots & 0b_2 \\
	\end{array}$$
	\\
	\item  Fix $\alpha$, and arbitrarily order the set $P_\alpha (6k+4) = \{[r(1),s(1)], [r(2), s(2)], \ldots , [r(t), s(t)]\}$.  For each $d = 0, 1, 2, \ldots , t-1$, we create a cycle $C_d$ as follows:
	$$\begin{array}{llll}
		0r(1) & 0s(1) & 1s(1+d) & 1r(1+d) \\
		1r(1+d) & 1s(1+d) & 2s(1+2d) & 2r(1+2d) \\
		2r(1+2d) & 2s(1+2d) & 0s(1+3d) & 0r(1+3d) \\
		\vdots \\
		2r(1-d) & 2s(1-d) & 0s(1) & 0r(1)
	\end{array}$$
	This cycle will always eventually reach the final block.  However, it may or may not also cover the intermediary blocks $\{1r(1-d), 1s(1-d), 2s(1), 2r(1)\}$ and $\{0r(1-d), 0s(1-d), 1s(1), 1r(1)\}$.  If $C_d$ misses these intermediary blocks, then we create a new cycle similar to $C_d$ but replacing $0,1,2$ with either $1,2,0$ or $2,0,1$, respectively. 
\end{enumerate}

To connect all of these cycles, we first focus on (3).  We create one long cycle by fixing $b_2=0$ and creating a cycle $C_{b_3}$ for each $b_3 \in \{0, 1, 2, \ldots , n-3\}$.  We connect all of these cycles at their points $00$ to create one long cycle $C$ that has every overlap of the type $$\{0, 1, 2\} \times \{0, 1, 2, \ldots , n-3\}.$$  Using this cycle $C$, we can connect every other cycle to $C$ to make one long cycle containing all quadruples.
\end{proof}

\begin{result}
	Let $(X, \mathcal{B})$ be an SQS($n$) that admits a 1-ocycle.  Then there exists an SQS($4n-6$) that admits a 1-ocycle.
\end{result}
\begin{proof}
	We use Construction \ref{3.5}.  To create cycles on these blocks, we do the following.
\begin{enumerate}
	\item  If $\mathcal{C}$ is a 1-overlap cycle on $\mathcal{B}$, then for each choice of $h, i \in \{0, 1\}$ we create the cycle $h \oplus i \oplus \mathcal{C}$.
	\\
	\item  We will combine these with the triples of type (3).  Note in particular that each block is completely determined by our choice of $2c^{(2)}_2 - \epsilon^{(2)}$ and $2c^{(2)}_3 + \ell^{(2)}$ from $\{0, 1, 2, \ldots , n-3\}$, where superscript (2) denotes a variable corresponding to a block of type (2), and similarly for superscript (3) and blocks of type (3).
	\\
	\item  Note that each block is completely determined by our choice of $2c^{(3)}_2-1-\epsilon^{(3)}$ and $2c^{(3)}_3 + 1-\ell^{(3)}$ from $\{0, 1, 2, \ldots , n-3\}$.  Fix $x \in \{0, 1, 2, \ldots , n-3\}$, and define the  cycles as follow, which alternates between blocks of type (2) (where $x = 2c^{(2)}_3+\ell^{(2)}$) and (3) (where $x = 2c^{(3)}_3 + 1 - \ell^{(3)}$).  
	
	We will connect pairs of blocks (one of type (2), one of type (3)) with $2c^{(2)}_2 - \epsilon^{(2)} = 2c^{(3)}_2-1-\epsilon^{(3)}$.  Note that this implies that $\epsilon^{(2)} \neq \epsilon^{(3)}$.  These blocks are connected to make short strings as shown by matching a block of type (2) and (3) in which $2c_2^{(2)}-\epsilon^{(2)} = 2c^{(3)}_2-1-\epsilon^{(3)}$ : $$ \begin{array}{llll}
		1\epsilon^{(2)}x, & A\ell^{(2)}, & 00(2c_1^{(2)}), & 01(2c_2^{(2)}-\epsilon^{(2)}) \\
		01(2c^{(3)}_2-1-\epsilon^{(3)}), & 00(2c_1^{(3)}+1), & A\ell^{(3)},& 1\epsilon^{(3)}x
	\end{array}$$
	
	To connect these two-block strings, we define $y = 2c_2^{(2)}-\epsilon^{(2)} = 2c^{(3)}_2-1-\epsilon^{(3)}$, and let $y$ range from $0$ up to $n-3$ to create a cycle covering all of these strings.
	
	$$\begin{array}{lll}
		10x & \cdots & 010 \\
		010 & \cdots & 11x \\
		11x & \cdots & 011 \\
		011 & \cdots & 10x \\
		10x & \cdots & 012 \\
		012 & \cdots & 11x \\
		\vdots \\
		01(n-4) & \cdots & 11x \\
		11x & \cdots & 01(n-3) \\
		01(n-3) & \cdots & 10x
	\end{array}$$  For each choice of $x$ we will have one cycle, and each cycle has length $2(n-2)$, covering a total of $2(n-2)^2$ blocks.
	\\
	\item  We will combine these with the triples of type (5).  Note in particular that each block is completely determined by our choice of $2c_2-\epsilon$ and $2c_3 + 1 - \ell$ from $\{0, 1, 2 , \ldots , n-3\}$.
	\\
	\item  We proceed in a manner similar to (3).  Note that each block is completely determined by our choice of $2c_2-1-\epsilon$ and $2c_3+\ell$ from $\{0, 1, 2, \ldots , n-3\}$.  Fix $x \in \{0, 1, 2, \ldots , n-3\}$, and define the following cycle, which alternates between blocks of type (4) (where $x=2c_3+1-\ell$) and (5) (where $x=2c_3 + \ell$).
	$$\begin{array}{lll}
		00x & \cdots & 110 \\
		110 & \cdots & 01x \\
		01x & \cdots & 111 \\
		111 & \cdots & 00x \\
		00x & \cdots & 112 \\
		112 & \cdots & 01x \\
		\vdots \\
		11(n-4) & \cdots & 01x \\
		01x & \cdots & 11(n-3) \\
		11(n-3) & \cdots & 00x
	\end{array}$$
	\\
	\item  We will combine these with the triples of type (8).
	\\
	\item  We will combine these with the triples of type (9).
	\\
	\item  Note that the first two terms in both (6) and (8) are the same.  For each choice of $2c_1 + \epsilon, 2c_2 - \epsilon \in \{0, 1, 2, \ldots , n-3\}$, we create the short 2-block cycle: $$h0(2c_1 + \epsilon) \cdots h1(2c_2 - \epsilon)$$ $$h1(2c_2-\epsilon) \cdots h0(2c_1 + \epsilon)$$ where the first block is of type (6) and the second block is of type (8).
	\\
	\item  Note that the first two terms in both (7) and (9) are the same.  For each choice of $2c_1-1+\epsilon, 2c_2 - \epsilon \in \{0, 1, 2, \ldots , n-3\}$ we create the short 2-block cycle:  $$h0(2c_1-1 + \epsilon) \cdots h1(2c_2-\epsilon)$$ $$h1(2c_2-\epsilon) \cdots h0(2c_1-1 + \epsilon)$$ where the first block is of type (7) and the second block is of type (9).
	\\
	\item  We use a similar method as in similar blocks in previous constructions.  Fix $\alpha$ and write $P_\alpha(k) = \{[r(1),s(1)], [r(2), s(2)], \ldots , [r(t), s(t)]\}$.  For each difference $d$ from 2 to $\lfloor \frac{t}{2} \rfloor$ we define several cycles.  For each difference, the pairs will be partitioned into sets of a certain size, say $s_d$.  The number of cycles defined will depend on whether $s_d$ is even or odd.
	
	For example, for difference 2, if $s_1$ is even then we have the four cycles (two for each choice of $h \in \{0, 1\}$) for each partition class: $$\begin{array}{lll}
		h0r(1) & \cdots & h1r(3) \\
		h1r(3) & \cdots & h0r(5) \\
		h0r(5) & \cdots & h1r(7) \\
		\vdots \\
		h1r(t-1) & \cdots & h0r(1)
	\end{array}$$ and $$\begin{array}{lll}
		h1r(1) & \cdots & h0r(3) \\
		h0r(3) & \cdots & h1r(5) \\
		h1r(5) & \cdots & h0r(7) \\
		\vdots \\
		h0r(t-1) & \cdots & h1r(1)
	\end{array}$$
	
	If $s_1$ is odd we have the two cycles (one for each choice of $h$):
	$$\begin{array}{lll}
		h0r(1) & \cdots & h1r(3) \\
		h1r(3) & \cdots & h0r(5) \\
		h0r(5) & \cdots & h1r(7) \\
		\vdots \\
		h0r(t-1) & \cdots & h1r(1) \\
		h1r(1) & \cdots & h0r(3) \\
		h0r(3) & \cdots & h1r(5) \\
		h1r(5) & \cdots & h0r(7) \\
		\vdots \\
		h0r(t-1) & \cdots & h1r(1)
	\end{array}$$
	We construct cycles in a similar manner for each partition set of each difference.
	
	All that remains are the blocks of differences 0 and 1.  We will construct two cycles, one for each choice of $h$ as follows:
	$$\begin{array}{llll}
		h0r(1) & h0s(1) & h1s(2) & h1r(2) \\
		h1r(2) & h1s(2) & h0r(2) & h0s(2) \\
		h0s(2) & h1s(1) & h1r(1) & h0r(2) \\
		\hline
		h0r(2) & h0s(2) & h1s(3) & h1r(3) \\
		h1r(3) & h1s(3) & h0r(3) & h0s(3) \\
		h0s(3) & h1s(2) & h1r(2) & h0r(3) \\
		\hline
		\vdots \\
		\hline
		h0r(t) & h0s(t) & h1s(1) & h1r(1) \\
		h1r(1) & h1s(1) & h0r(1) & h0s(1) \\
		h0s(1) & h1s(t) & h1r(t) & h0r(1) \\
	\end{array}$$
\end{enumerate}
Now we must connect all of these cycles.  Note that the cycles described in (3) and (5) contains as overlap points all points of the form $11y$ for $y \in \{0, 1, 2, \ldots , n-3\}$, regardless of our choice of $x$.  Thus we can connect all of these cycles together to make one cycle.  This one long cycle contains as overlap points all points of the form $hiy$ for $h \in \{0, 1\}$, $i \in \{0, 1\}$, and $y \in \{0, 1, 2, \ldots , n-3\}$.  Thus we can connect everything to this cycle, since no points $A\ell$ are used as overlap points.
\end{proof}

\begin{result}
	Let $(X, \mathcal{B})$ be an SQS($n$) that admits a 1-ocycle.  Then there exists an SQS($12n-10$) that admits a 1-ocycle.
\end{result}
\begin{proof}
	We begin by proving that the construction for an SQS(38) admits a 1-ocycle.  To construct a 1-overlap cycle, we look at each type of block separately.
\begin{enumerate}
	\item  Since we have a 1-overlap cycle for the SQS(14), clearly we can construct one cycle for each $i$ on these types of blocks. \\
	\item  Fix $b_1$ and $b_2$ and construct a short 2-block cycle by changing $h$ from 0 to 1. 
	$$\begin{array}{llll}
		0b_1 & A0 & 2b_3 & 1b_2 \\
		1b_2 & 2(b_3+3) & A1 & 0b_1 
	\end{array}$$
	\\
	\item  Fix $\{b_2, b_3\} = \{x,y\}$ with $x \neq y$ and create the cycle:
	$$\begin{array}{lll}
		0x & \cdots & 1y \\
		1y & \cdots & 2x \\
		2x & \cdots & 0y \\
		0y & \cdots & 1x \\
		1x & \cdots & 2y \\
		2y & \cdots & 0x
	\end{array}$$ 
	When we have $b_2 = b_3 = x$ we have the shorter cycle:
	$$\begin{array}{lll}
		0x & \cdots & 1x \\
		1x & \cdots & 2x \\
		2x & \cdots & 0x
	\end{array}$$
	\\
	\item  We will combine blocks of type (4) with blocks of type (6). \\
	\item  We will combine blocks of type (5) with blocks of type (7). \\
	\item  Note that the first term and the third term in the block of type (4) are the same as those for the blocks of type (6).  Using this, we create short two-block cycles by fixing $\epsilon$ and connecting the blocks as shown:
	$$ij \cdots (i+2)(6\epsilon -2j+1) \cdots ij.$$
	\\
	\item  Note that the first term and the last term in the block of type (5) are the same as those for the blocks of type (7).  Using this, we create short two-block cycles by fixing $\epsilon$ and connecting the blocks as shown: $$ij \cdots (i+2)(6 \epsilon - 2j-2) \cdots ij.$$
	\\
	\item  For each choice of $\epsilon \in \{0,1\}$, we have two blocks that contain the points $ij, i(j+6)$ of type (8).  We connect these as shown to create short two-block cycles: $$ij \cdots i(j+6) \cdots ij.$$
	\\
	\item  We will combine these with the blocks of type (10). \\
	\item  Fix $g \in \{0, 1, 2, 3, 4, 5\}$, $\epsilon \in \{0, 1\}$ and fix $i \in \{0, 1, 2\}$.  Then we have two choices for $i' \in \{0, 1, 2\} \setminus \{i\}$, so we have narrowed our consideration to 2 blocks of type (9) and 2 blocks of type (10).  In any order, we list them in a 1-overlap cycle as: $$i(2g+3\epsilon) \cdots i(2g+6+3\epsilon) \cdots i(2g+3 \epsilon) \cdots i (2g+6+3 \epsilon) \cdots i(2g+3\epsilon )$$
	\\
	\item  Fix $i,j$.  This restricts us to 4 blocks - one for each $e \in \{0, 1, 2, 3\}$.  Form a 1-overlap cycle as follows: $$ij \cdots i(j+1) \cdots ij \cdots i(j+1) \cdots ij.$$
	\\
	\item  Same as (11).
	\\
	\item  Same as (11).
	\\
	\item  Fix $i,i' \in \{0,1,2\}$ (distinct), and fix $\alpha \in \{4,5\}$.  Order the pairs in $P_\alpha(6)$ arbitrarily as $\{[r(1), s(1)], [r(2), s(2)], \ldots , [r(t), s(t)]\}$.  For each difference $d = 1, 2, \ldots , \lfloor t/2 \rfloor$, we construct the string $S_1$ as follows:
	$$\begin{array}{llll}
		i r(1) & i s(1) & i' s(1+d) & i'r(1+d) \\
		i'r(1+d) & i' s(1+d) & i s(1+2d) & i r(1+2d) \\
		i r(1+2d) & i s(1+2d) & i' s(1+3d) & i' r(1+3d) \\
		\vdots
	\end{array}$$
	Continue this string until we arrive at a block that ends with either $i r(1)$ or $i' r(1)$.  We construct $S_2$ from $S_1$ by swapping $i$ and $i'$ everywhere.  If $S_1$ ends in $ir(1)$, then $S_2$ ends in $i' r(1)$, and both are cycles.  If $S_1$ ends in $i' r(1)$, then $S_2$ ends in $i r(1)$ and we can connect $S_1$ and $S_2$ to make one cycle.  If $d$ divides $ t$, then we repeat the procedure, replacing $[r(1), s(1)]$ with a pair from $P_\alpha(6)$ that was not used in creating $S_1$ and $S_2$.
\end{enumerate}

To connect all of these cycles, we look to the blocks of type (2).  Fix $b_1$, and connect all of the short cycles corresponding to this $b_1$ together at the point $0b_1$.  Note that this cycle has as overlap points $0b_1$ and every point $1j$ for $j \in \{0, 1, 2, \ldots , 11\}$.  Do this for each choice of $b_1$, and then each of these cycles contains the point $10$ as an overlap, so they can all be connected.  From this, we have constructed one cycle that contains every point $ij$ with $i, j \in \{0, 1, 2, \ldots , 11\}$ as an overlap point.  All other cycles can be connected to this.
\\

\textbf{The general $\textbf{SQS}\mathbf{(n) \rightarrow SQS(12n-10)}$ construction:}  To construct an overlap cycle, we look at each type of block separately.
\begin{enumerate}
 \item Since we can construct a 1-ocycle on $\mathcal{B}(14)$, we can construct a 1-ocycle on $i \oplus \mathcal{B}(14)$ for each choice of $i$.
\\
\item  These blocks are completely determined by choosing $b_1, b_2 \in \{0, 1, \ldots , 11\}$ and $h \in \{0,1\}$.  If we fix $b_1, b_2$, then we can connect the two blocks identified as follows: $$ub_1 \cdots vb_2 \cdots ub_1.$$  We do this for each choice of $b_1$ and $b_2$ to make many short 2-block cycles.
\\
\item  These blocks have structure similar to their corresponding blocks from the SQS(38).  Therefore, we make cycles in exactly the same way for blocks of type (4) - (7) in the SQS(38).  For blocks of type (3) in the SQS(38), we can use the same structure, since only the last 2 points are used as overlaps.  In this SQS($12n-10$), this corresponds to the points $w\alpha_3, w\alpha_4$, and so we can create cycles in the same way.
\\
\item  We can make cycles in the same way as in the SQS(38) for blocks of type (8) - (14).
\end{enumerate}
To connect all of these cycles, we look to the blocks of type (2).  Fix $u, b_1$, and let everything else vary.  These many short cycles can all be connected at the point $ub_1$, and the cycle created contains every point as an overlap point, except $A0, A1$.  We can connect all other cycles to this one.
\end{proof}

Putting all of these constructions together with the base cases, we get the following result.
\begin{result}
	For all $n \equiv 2,4 \pmod 6$ with $n \geq 8$, there exists an SQS($n$) that admits a 1-ocycle.
\end{result}
\begin{proof}
	These constructions, together with Theorem \ref{Hanani} and the corresponding base cases, we have a 1-ocycle for each order.
\end{proof}

\section{The Construction of 2-Ocycles}

Several other constructions for Steiner quadruple systems exist.  In this section, we explore one of these constructions and find 2-ocycles corresponding to it.  Note that this is best possible for Steiner quadruple systems, as an overlap of size three is not possible.

\begin{defn} \emph{\cite{Cameron}, p. 121.} A \textit{tournament schedule} for $n$ teams is an arrangement of all pairs of teams into the minimum number of rounds ($n-1$ if $n$ is even, and $n$ if $n$ is odd).
\end{defn}

\begin{const}\label{TConst} \textbf{Tournament Construction}, \emph{\cite{Cameron}, p. 119. }\end{const}  
\begin{proof}[Construction]  For $n$ even, we can construct a tournament as follows.  List $n-1$ points evenly around a circle and add a point to the center.  We now have $n$ points.  Match the center point to one of the outer points, and create a matching on the remaining points by taking all pairs perpendicular to the center point's match.  This gives one tournament round.  Rotating the center point's match around the outer circle gives $n-1$ rounds.

For $n$ odd, we use the same construction as described above using $n+1$ points, but discard the center point after creating the matching.  This leaves the center point's match as the point that does not participate in the round.
\end{proof}

For example, when we have 6 points, we have the following rounds:

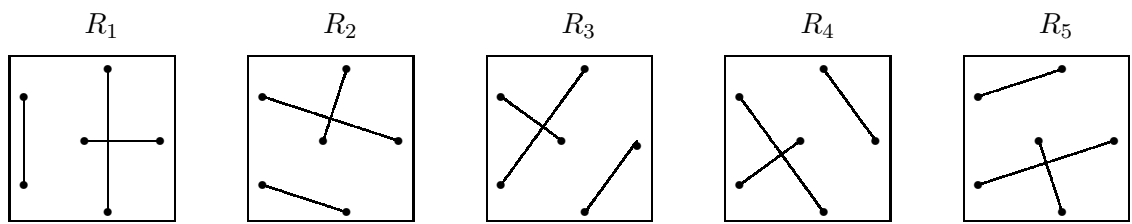
\setlength{\unitlength}{10mm}
\begin{picture}(2,3)
	\qbezier(0,0)(0,0)(0,2.2)
	\qbezier(0,0)(0,0)(2.2,0)
	\qbezier(2.2,0)(2.2,0)(2.2,2.2)
	\qbezier(0,2.2)(0,2.2)(2.2,2.2)
	\put(1.0,2.5){$R_1$}
	\put(1,1.07){\circle*{.1}}
	\put(1.3090,2.0211){\circle*{.1}}
	\put(0.1910,1.6578){\circle*{.1}}
	\put(0.1910,0.4822){\circle*{.1}}
	\put(1.3090,0.1189){\circle*{.1}}
	\put(2,1.07){\circle*{.1}}
	\qbezier(1,1.07)(2,1.07)(2,1.07)
	\qbezier(0.1910,1.6578)(0.1910,1.6578)(0.1910,0.4822)
	\qbezier(1.3090,2.0211)(1.3090,2.0211)(1.3090,0.1189)
\end{picture}
\hspace{10 mm}
\begin{picture}(2,3)
	\qbezier(0,0)(0,0)(0,2.2)
	\qbezier(0,0)(0,0)(2.2,0)
	\qbezier(2.2,0)(2.2,0)(2.2,2.2)
	\qbezier(0,2.2)(0,2.2)(2.2,2.2)
	\put(1.0,2.5){$R_2$}
	\put(1,1.07){\circle*{.1}}
	\put(1.3090,2.0211){\circle*{.1}}
	\put(0.1910,1.6578){\circle*{.1}}
	\put(0.1910,0.4822){\circle*{.1}}
	\put(1.3090,0.1189){\circle*{.1}}
	\put(2,1.07){\circle*{.1}}
	\qbezier(1,1.07)(1.3090,2.0211)(1.3090,2.0211)
	\qbezier(0.1910,1.6578)(0.1910,1.6578)(2,1.07)
	\qbezier(0.1910,0.4822)(0.1910,0.4822)(1.3090,0.1189)
\end{picture}
\hspace{10 mm}
\begin{picture}(2,3)
	\qbezier(0,0)(0,0)(0,2.2)
	\qbezier(0,0)(0,0)(2.2,0)
	\qbezier(2.2,0)(2.2,0)(2.2,2.2)
	\qbezier(0,2.2)(0,2.2)(2.2,2.2)
	\put(1.0,2.5){$R_3$}
	\put(1,1.07){\circle*{.1}}
	\put(1.3090,2.0211){\circle*{.1}}
	\put(0.1910,1.6578){\circle*{.1}}
	\put(0.1910,0.4822){\circle*{.1}}
	\put(1.3090,0.1189){\circle*{.1}}
	\put(2,1){\circle*{.1}}
	\qbezier(1,1.07)(0.1910,1.6578)(0.1910,1.6578)
	\qbezier(1.3090,2.0211)(1.3090,2.0211)(0.1910,0.4822)
	\qbezier(1.3090,0.1189)(1.3090,0.1189)(2,1.07)
\end{picture}
\hspace{10 mm}
\begin{picture}(2,3)
	\qbezier(0,0)(0,0)(0,2.2)
	\qbezier(0,0)(0,0)(2.2,0)
	\qbezier(2.2,0)(2.2,0)(2.2,2.2)
	\qbezier(0,2.2)(0,2.2)(2.2,2.2)
	\put(1.0,2.5){$R_4$}
	\put(1,1.07){\circle*{.1}}
	\put(1.3090,2.0211){\circle*{.1}}
	\put(0.1910,1.6578){\circle*{.1}}
	\put(0.1910,0.4822){\circle*{.1}}
	\put(1.3090,0.1189){\circle*{.1}}
	\put(2,1.07){\circle*{.1}}
	\qbezier(1,1.07)(0.1910,0.4822)(0.1910,0.4822)
	\qbezier(1.3090,0.1189)(1.3090,0.1189)(0.1910,1.6578)
	\qbezier(1.3090,2.0211)(1.3090,2.0211)(2,1.07)
\end{picture}
\hspace{10 mm}
\begin{picture}(2,3)
	\qbezier(0,0)(0,0)(0,2.2)
	\qbezier(0,0)(0,0)(2.2,0)
	\qbezier(2.2,0)(2.2,0)(2.2,2.2)
	\qbezier(0,2.2)(0,2.2)(2.2,2.2)
	\put(1.0,2.5){$R_5$}
	\put(1,1.07){\circle*{.1}}
	\put(1.3090,2.0211){\circle*{.1}}
	\put(0.1910,1.6578){\circle*{.1}}
	\put(0.1910,0.4822){\circle*{.1}}
	\put(1.3090,0.1189){\circle*{.1}}
	\put(2,1.07){\circle*{.1}}
	\qbezier(1,1.07)(1.3090,0.1189)(1.3090,0.1189)
	\qbezier(1.3090,2.0211)(1.3090,2.0211)(0.1910,1.6578)
	\qbezier(0.1910,0.4822)(0.1910,0.4822)(2,1.07)
\end{picture}
\\
\newpage
When we have 5 points, we have the following rounds:

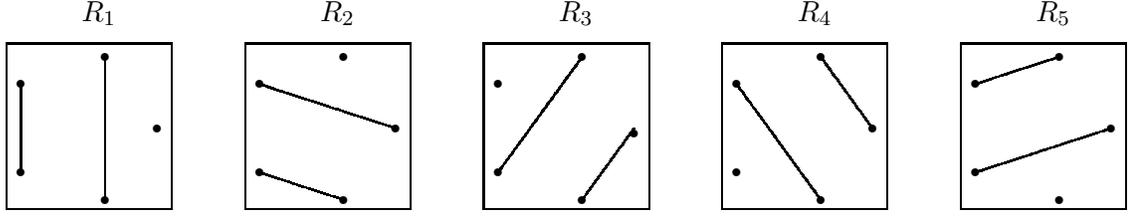
\setlength{\unitlength}{10mm}
\begin{picture}(2,3)
	\qbezier(0,0)(0,0)(0,2.2)
	\qbezier(0,0)(0,0)(2.2,0)
	\qbezier(2.2,0)(2.2,0)(2.2,2.2)
	\qbezier(0,2.2)(0,2.2)(2.2,2.2)
	\put(1.0,2.5){$R_1$}
	\put(1.3090,2.0211){\circle*{.1}}
	\put(0.1910,1.6578){\circle*{.1}}
	\put(0.1910,0.4822){\circle*{.1}}
	\put(1.3090,0.1189){\circle*{.1}}
	\put(2,1.07){\circle*{.1}}
	\qbezier(0.1910,1.6578)(0.1910,1.6578)(0.1910,0.4822)
	\qbezier(1.3090,2.0211)(1.3090,2.0211)(1.3090,0.1189)
\end{picture}
\hspace{10 mm}
\begin{picture}(2,3)
	\qbezier(0,0)(0,0)(0,2.2)
	\qbezier(0,0)(0,0)(2.2,0)
	\qbezier(2.2,0)(2.2,0)(2.2,2.2)
	\qbezier(0,2.2)(0,2.2)(2.2,2.2)
	\put(1.0,2.5){$R_2$}
	\put(1.3090,2.0211){\circle*{.1}}
	\put(0.1910,1.6578){\circle*{.1}}
	\put(0.1910,0.4822){\circle*{.1}}
	\put(1.3090,0.1189){\circle*{.1}}
	\put(2,1.07){\circle*{.1}}
	\qbezier(0.1910,1.6578)(0.1910,1.6578)(2,1.07)
	\qbezier(0.1910,0.4822)(0.1910,0.4822)(1.3090,0.1189)
\end{picture}
\hspace{10 mm}
\begin{picture}(2,3)
	\qbezier(0,0)(0,0)(0,2.2)
	\qbezier(0,0)(0,0)(2.2,0)
	\qbezier(2.2,0)(2.2,0)(2.2,2.2)
	\qbezier(0,2.2)(0,2.2)(2.2,2.2)
	\put(1.0,2.5){$R_3$}
	\put(1.3090,2.0211){\circle*{.1}}
	\put(0.1910,1.6578){\circle*{.1}}
	\put(0.1910,0.4822){\circle*{.1}}
	\put(1.3090,0.1189){\circle*{.1}}
	\put(2,1){\circle*{.1}}
	\qbezier(1.3090,2.0211)(1.3090,2.0211)(0.1910,0.4822)
	\qbezier(1.3090,0.1189)(1.3090,0.1189)(2,1.07)
\end{picture}
\hspace{10 mm}
\begin{picture}(2,3)
	\qbezier(0,0)(0,0)(0,2.2)
	\qbezier(0,0)(0,0)(2.2,0)
	\qbezier(2.2,0)(2.2,0)(2.2,2.2)
	\qbezier(0,2.2)(0,2.2)(2.2,2.2)
	\put(1.0,2.5){$R_4$}
	\put(1.3090,2.0211){\circle*{.1}}
	\put(0.1910,1.6578){\circle*{.1}}
	\put(0.1910,0.4822){\circle*{.1}}
	\put(1.3090,0.1189){\circle*{.1}}
	\put(2,1.07){\circle*{.1}}
	\qbezier(1.3090,0.1189)(1.3090,0.1189)(0.1910,1.6578)
	\qbezier(1.3090,2.0211)(1.3090,2.0211)(2,1.07)
\end{picture}
\hspace{10 mm}
\begin{picture}(2,3)
	\qbezier(0,0)(0,0)(0,2.2)
	\qbezier(0,0)(0,0)(2.2,0)
	\qbezier(2.2,0)(2.2,0)(2.2,2.2)
	\qbezier(0,2.2)(0,2.2)(2.2,2.2)
	\put(1.0,2.5){$R_5$}
	\put(1.3090,2.0211){\circle*{.1}}
	\put(0.1910,1.6578){\circle*{.1}}
	\put(0.1910,0.4822){\circle*{.1}}
	\put(1.3090,0.1189){\circle*{.1}}
	\put(2,1.07){\circle*{.1}}
	\qbezier(1.3090,2.0211)(1.3090,2.0211)(0.1910,1.6578)
	\qbezier(0.1910,0.4822)(0.1910,0.4822)(2,1.07)
\end{picture}
\\
\\
\begin{const}\label{nTO2n}  \textbf{Recursive SQS Construction} \emph{\cite{Cameron}, p. 121.} \end{const}  
 Let $(X, \mathcal{B})$ be a SQS of order $n \geq 2$.  Take a disjoint copy of $(X', \mathcal{B}')$ of this system.  Take a tournament schedule on $X$ with rounds $R_1, \ldots , R_{n-1}$, and one on $X'$ with rounds $R'_1, \ldots , R'_{n-1}$.  This is possible since $n \equiv 2, 4 \pmod 6$, which implies that $n$ is even.  Define:
\begin{eqnarray*}
	Y & = & X \cup X' \\
	\mathcal{C} & = & \mathcal{B} \cup \mathcal{B}' \cup \mathcal{R}
\end{eqnarray*}	
where $\mathcal{R}$ contains all sets $\{x, y, z', w'\}$ such that $x,y \in X$ and $z',w' \in X'$, and there exists $i \in [n-1]$ such that $\{x,y\} \in R_i$ and $\{z', w'\} \in R'_i$.  Then $(Y, \mathcal{C})$ is an SQS of order $2n$.

\begin{thm}\label{ProvenTO2n}
	Construction \ref{nTO2n} produces an SQS$(2n)$.
\end{thm}
\begin{proof}  The total number of blocks in an SQS($n$) is given by $$\frac{{n \choose 3}}{{4 \choose 3}} = \frac{n(n-1)(n-2)}{24}.$$  To show that this construction produces an SQS($2n$), we need only show two things:  (a)  every triple appears in some block, and (b)  we have the correct number of blocks, i.e. $\frac{2n(2n-1)(2n-2)}{24}$ blocks.

For (a), we consider two cases.  First, consider a triple $\{a,b,c\}$ with either $\{a,b,c\} \subset X$ or $\{a,b,c\} \subset X'$.  Then this triple is contained in a block in either $(X, \mathcal{B})$ or $(X', \mathcal{B}')$.  

Next, consider a triple $\{a, b', c'\}$ with $a \in X$ and $b',c' \in X'$.  Then we locate the tournament round $R'_i$ in which teams $b'$ and $c'$ are paired.  Then there is some $z \in X$ such that $\{a,z\}$ are paired in round $R_i$.  The construction then produces the block $\{a,z,b',c'\}$.  The same argument holds for triples $\{a',b,c\}$ with $a' \in X'$ and $b,c \in X$.

Finally, for (b) we count the total number of blocks.  Our construction uses all of the blocks from $(X, \mathcal{B})$ and $(X', \mathcal{B}')$, the sum of which is \begin{equation} \frac{n(n-1)(n-2)}{12}.\end{equation}  Then we must count all of the blocks created from the tournament.  For each $i \in [n-1]$, we have all combinations of pairs from $R_i \times R'_i$, which gives us $ \frac{n^2}{4} $ different blocks.  Since there are $n-1$ choices for $i$, the total number of blocks created using the tournament is: \begin{equation} \frac{n^2(n-1)}{4}. \end{equation}  Summing (1) and (2), we have that the total number of blocks in our SQS($2n$) is:
\begin{eqnarray*}
	\frac{n(n-1)(n-2)}{12} + \frac{n^2(n-1)}{4}
	& = & \frac{2n(n-1)(4n-2)}{24} \\
	& = & \frac{2n(2n-1)(2n-2)}{24}
\end{eqnarray*}
Thus we have the correct number of blocks, and so $(Y, \mathcal{C})$ is indeed an SQS($2n$).
\end{proof}

We now prove that we can use this construction to find 2-ocycles.

\begin{result}\label{2OC} Let $(X, \mathcal{B})$ be a SQS of order $n \geq 2$ that admits a 2-ocycle.  Then there exists a SQS of order $2n$ that admits a 2-ocycle if $n \equiv 0 \pmod 4$ and $n > 13$.
\end{result}

\begin{proof} Let $O$ be the 2-ocycle for $(X, \mathcal{B})$, and let $O'$ be the 2-ocycle for the disjoint copy of $(X, \mathcal{B})$ created in Construction \ref{nTO2n}.  We identify the points from $X$ with $\mathbb{Z}/n \mathbb{Z}$.  

The blocks from $\mathcal{R}$ are partitioned into $n-1$ classes:  $R_1\times R'_1, R_2\times R'_2, \ldots, R_{n-1} \times R'_{n-1}$.  Each of these $R_i$s and $R'_i$s consists of $\frac{n}{2}$ disjoint pairs of points.  For each $i \in [n-1]$, we construct a complete bipartite graph $G_i$ with parts $R_i$ and $R'_i$.  More specifically, each vertex in this graph represents one of the pairs of points.  Then, each edge in $G_i$ represents a quadruple from $\mathcal{R}$.  More importantly, an Euler tour in this graph corresponds to a 1-ocycle $C$ on pairs of pairs of points (one pair from $R_i$ and one pair from $R'_i$), which is a 2-overlap cycle on the quadruples from $R_i \times R'_i \subset \mathcal{R}$.  Since $G_i$ is clearly connected, such an Euler tour will exist if and only if $\frac{n}{2}$ is even.

Finally, we need only show that all of these cycles can be connected.  Fix a pair of points $\{x,y\} \in X$, and suppose that $\{x,y\} \in R_i$.  Let $B \in \mathcal{B}$ be a block containing $\{x,y\}$, say $\{x,y,z,w\}$.  In $O$, we must have (without loss of generality) either $\{x,y\}$ or $\{y,z\}$ as an overlap pair.  We must have $\{y,z\} \in R_j$ for some $j \in [n-1]$.  Depending on whether $\{x,y\}$ or $\{y,z\}$ is an overlap pair, we can attach one of the two smaller 2-overlap cycles, $R_i \times R'_i$ or $R_j \times R'_j$.  Repeating this idea, we may attach all but at most one of our cycles from $\mathcal{R}$, if one of the smaller 2-overlap cycles is repeatedly not attached.  So all that remains is to connect this one remaining cycle from $\mathcal{R}$;  suppose that this cycle corresponds to $R_k \times R'_k$.

Suppose for a contradiction that no overlap pair in $O$ is a pair from $R_k$.  Then we must have that all overlap pairs are points that compete in some round other than $R_k$.  Note that every pair must compete in some round.  The total number of overlap pairs in $O$ is equal to the number of blocks in $(X, \mathcal{B})$, or $\frac{n(n-1)(n-2)}{24}$.  The total number of pairs per round is $\frac{n}{2}$, and we may select pairs from any round except for $R_k$, giving us a total pool of $\frac{n(n-2)}{2}$ pairs.  Note that:
\begin{eqnarray*}
	13 & < & n \\
	12 & < & n-1 \\
	\frac{n(n-2)}{2} & < & \frac{n(n-1)(n-2)}{24}
\end{eqnarray*}
Thus the number of overlap pairs in $O$ exceeds our number of choices of pairs when we eliminate $R_k$, and so some overlap pair in $O$ must be represented as a pair in $R_k$.  Thus we can attach the 2-overlap cycle for round $R_k$ to our existing cycle.

Repeating the same argument for $O'$ allows us to attach $O'$ to at least one cycle of type $R_i \times R'_i$.  Thus we obtain a 2-overlap cycle covering $(Y, \mathcal{C})$.  
\end{proof}

\section{Future Work}

\setcounter{equation}{0}

Result \ref{2OC} shows a 2-ocycle for an SQS($2n$), given a 2-ocycle for an SQS($n$).  Can we find 2-ocycles that correspond to Hanani's SQS constructions?

\end{document}